\documentclass[12pt,a4paper,reqno]{amsart}
\usepackage[vmargin=2cm, hmargin=2cm]{geometry}
\usepackage{amsfonts,amsthm,amsmath,amssymb,amscd}
\usepackage{bbm}
\usepackage{graphicx}

\usepackage{enumerate}
\theoremstyle{definition}
\newtheorem{dfn}{Definition}
\newtheorem{defin}[dfn]{Definition}
\newtheorem{rem}[dfn]{Remark}

\theoremstyle{plain}
\newtheorem{lem}[dfn]{Lemma}
\newtheorem{thm}[dfn]{Theorem}

\newtheorem{prop}[dfn]{Proposition}
\newtheorem{ex}[dfn]{Example}
\newtheorem{cor}[dfn]{Corollary}

\newcommand{\Z}{{\mathbb Z}}
\newcommand{\R}{{\mathbb R}}
\newcommand{\N}{{\mathbb N}}

\DeclareMathOperator{\diam}{diam}

\numberwithin{equation}{section}

\title[Slow recurrence]{Estimating Hausdorff measure for Denjoy maps}


\begin{document}

\author{{\L}ukasz Pawelec}
\address{{\L}ukasz Pawelec, Instytut Matematyczny Polskiej Akademii Nauk, ul.~\'{S}niadeckich 8, 00-956 Warszawa, Poland}
\email{lpawelec@impan.pl}

\author{Mariusz Urba\'{n}ski}
\address{Mariusz Urba\'{n}ski, University of North Texas, Department of Mathematics, 1155 Union Circle \#311430, Denton, TX 76203-5017,
USA}
\email{urbanski@unt.edu}

\date{}
\begin{abstract}
By employing the recurrence method worked out in \cite{Pawpre}, we provide effective lower estimates of the proper--dimensional Hausdorff measure of  minimal sets of circle homeomorphisms that are not conjugate to any rotation. 
\end{abstract}
\maketitle

\section{Introduction}

In this paper we deal with orientation preserving homeomorphisms and, more specifically, diffeomorphisms of the unit circle~$\mathcal{S}^1$. A classical Poincar\'e's result states that for every such homeomorphism $f$ there is a unique number $\alpha\in[0,1)$ such that the map $f\colon \mathcal{S}^1\to\mathcal{S}^1$ and the rotation $R_\alpha(t):=t+\alpha\, (\!\!\!\!\mod\! 1)$ are semi-conjugate, i.e. $h\circ f = R_\alpha\circ h$, for some continuous surjection $h\colon\mathcal{S}^1\to\mathcal{S}^1$. The number $\alpha$ is called the rotation number of $f$. It can be defined more explicitly without bringing up semi-conjugacies but we do not need such definition in our paper. We do need semi-conjugacy.

Much later, Denjoy proved that if the derivative of $f$ is of bounded variation, then there is a (full) conjugacy between $f$ and $R_\alpha$, i.e. the map $h$ is a homeomorphism. However, if the derivative $f'$ is only $\delta$--H\"older continuous for some $0<\delta<1$, then the semi-conjugacy cannot be made into a full conjugacy. We then call the homeomorphism (or diffeomorphism) $f$ a \emph{Denjoy map}. 

The reason for the non-existence of the full conjugacy is the presence of  wandering intervals. Then the complement of the (forward and backward) orbit of that interval is the minimal set for the considered system $f\colon \mathcal{S}^1\to\mathcal{S}^1$. It is a compact, perfect, nowhere dense subset of $\mathcal{S}^1$, thus a Cantor set. In this paper we study the Hausdorff dimension and the Hausdorff measure of this unique minimal set of $f$. We will denote it by $\Omega(f)$ or, even simpler, just by $\Omega$.

Up to our knowledge, the, up to date, best estimates of the Hausdorff and box-counting dimensions of the minimal set of a Danjoy map were obtained by B.~Kra and J.~Schmeling in \cite{KS}. They proved a general lower bound of the box-counting dimension, depending only on the smoothness of the derivative of $f$, namely that $\dim_B(\Omega(f)) \geq\delta$. They also showed that the lower bound of the Hausdorff dimension depends additionally on the Diophantine class $\nu\geq1$ of the rotation number $\alpha$; see Section \ref{sec:def} for the definition. More precisely, they proved that $\dim_H(\Omega)\geq \delta/\nu$. Moreover, for the case of so-called \emph{classical} Denjoy maps (see Example \ref{ex:classic}), they showed that both the Hausdorff and box-counting dimensions are in fact equal to their lower bounds, resp. $\delta/\nu$ and $\delta$, and so they coincide (only) for $\nu=1$. 

In this paper we push these results further. We apply a rather non-standard method of estimating the Hausdorff measure from below developed by the first named author in \cite{Pawpre}. There it is shown that given a continuous self-map $T\colon X\to X$ of a compact metric space $X$ preserving a Borel probability measure $\mu$, a good knowledge of recurrence properties of the map $T$ leads to lower bounds of the density of $\mu$  with respect to a Hausdorff measure $H$; in the case when $\mu$ is absolutely continuous with respect to $H$, this density is the Radon-Nikodym derivative $d\mu/dH$. This bound entails in turn bounds on $H(X)$. 

Applying this method we arrive at an effective bound from below on the Hausdorff measure, in the relevant Hausdorff dimension, of the minimal set of a diffeomorphism $f\colon \mathcal{S}^1\to\mathcal{S}^1$. It depends on the Diophantine properties of the rotation number and on the lengths of the  of the wandering intervals of $f$, which as we know, are the connected components of $\mathcal{S}^1\setminus\Omega(f)$. See Theorem \ref{520221110} for a precise statement. We would like to emphasize that the estimates we obtain are effectively computable as we show in several examples in the last section of the paper.

Furthermore, our results allow us to recover the lower bounds on the Hausdorff dimension of $\Omega$ obtained in \cite{KS}. In addition, for $\nu=1$ we provide a simple upper estimate of the Hausdorff dimension of $\Omega$ and it coincides with the lower bound. In some cases we can even get a meaningful upper estimate of the Hausdorff measure. For example, our results show that for the \emph{classical} Denjoy map the relevant Hausdorff measure is finite and positive and we provide with concrete bounds.

%\vskip5cm

%In the recent years there has been a growing interest in the topic of quantitative recurrence. There is currently a number of ways in which we may estimate the \emph{speed} of recurrence. The historically first one comes from the paper by M. Boshernitzan \cite{Bosh}. It asserts that if $(X,d)$ is a separable metric space and $(T,\mu)$ is any transformation preserving a Borel, probability measure, then 
%\begin{equation}\label{wstep}
%\liminf_{n\to \infty} \;n^{1/\beta}d(T^n(x),x) < +\infty,
%\end{equation} 
%if $H_\beta(X) <+\infty$. In addition, the above lower limit vanishes if $H_\beta(X)=0$.  

\section{Setting, Definitions and some Tools Used}\label{sec:def}
In this section we introduce notation and bring up the tools which we will use in the subsequent sections. Throughout the paper we always assume that $\alpha\in (0,1)$, the rotation number of a considered circle diffeomorphism, is irrational, and that this diffeomorphism is exactly $\mathcal{C}^{1+\delta}$ smooth with some $\delta\in(0,1)$. We denote by $\|x\|$ is the norm (arc-wise distance from the rightmost point) of a point $x$ on the  circle. Lastly, we always assume that the total length of a circle is equal to $1$.

\subsection{Construction of Denjoy Maps} \ \\
We start with recalling briefly the standard construction of Denjoy maps. The notation we implement here follows \cite{KS}. A precise construction may be found e.g. in \cite{KH}.

\begin{dfn}\label{def:denjoy}
Given a number $\delta\in(0,1)$, a \emph{Denjoy sequence} of class $\delta$ is any sequence $\{\ell_n\}_{n\in\Z}$ of non-negative numbers satisfying the following two conditions.
\begin{equation}\label{eq:denj1}
\sum_{n\in \Z} \ell_n =1 
\end{equation}
and
\begin{equation}\label{eq:denj2}
\liminf_{n\to\pm\infty} \frac{\ln|\ell_n-\ell_{n+1}|}{\ln\ell_n}
=1+\delta.
\end{equation}
\end{dfn}
\begin{ex} \label{ex:classic}
The \emph{classical Denjoy sequence} of class $\delta$ is given by the following formula.
\begin{equation}
\ell_n=c_\delta (|n|+1)^{-1/\delta}, \  n\ge 1,
\label{eq:denjcl}
\end{equation}
where $c_\delta$ is the normalizing constant making the sum $
\sum_{n\in\Z}\ell_n$ equal to $1$.

Observe that the \emph{classical} sequence trivially satisfies the requirements of Definition~\ref{def:denjoy}. Such a sequence may be considered as a \emph{best one} in the class of Denjoy sequences.

Also, using Riemann's $\zeta$ function we may state the value of the normalising constant: 
\begin{equation}\label{eq:cdelta}
c_\delta = 2\zeta\Big(\frac{1}{\delta}\Big)-1.\end{equation}
\end{ex}

%Let $\alpha\in(0,1)$ be irrational and let $\delta\in(0,1)$. 
Now, fix a $\delta\in(0,1)$ and take any Denjoy sequence $(\ell_n)_{n\in \Z}$ of class $\delta$. To obtain a Denjoy map, we start with a circle of length $1$ with a map that is the rotation by $\alpha$. Then, we \emph{blow up} subsequently every point on the orbit of $0$, i.e. point of the form $\pm n\alpha$, $n\in \N_0$, to an open interval $J_n$ of length $\ell_n$, rescaling at every step the remainder of the circle to have its length equal equal to $1$.

\begin{figure}[h]
	\centering
\includegraphics[scale=0.9]{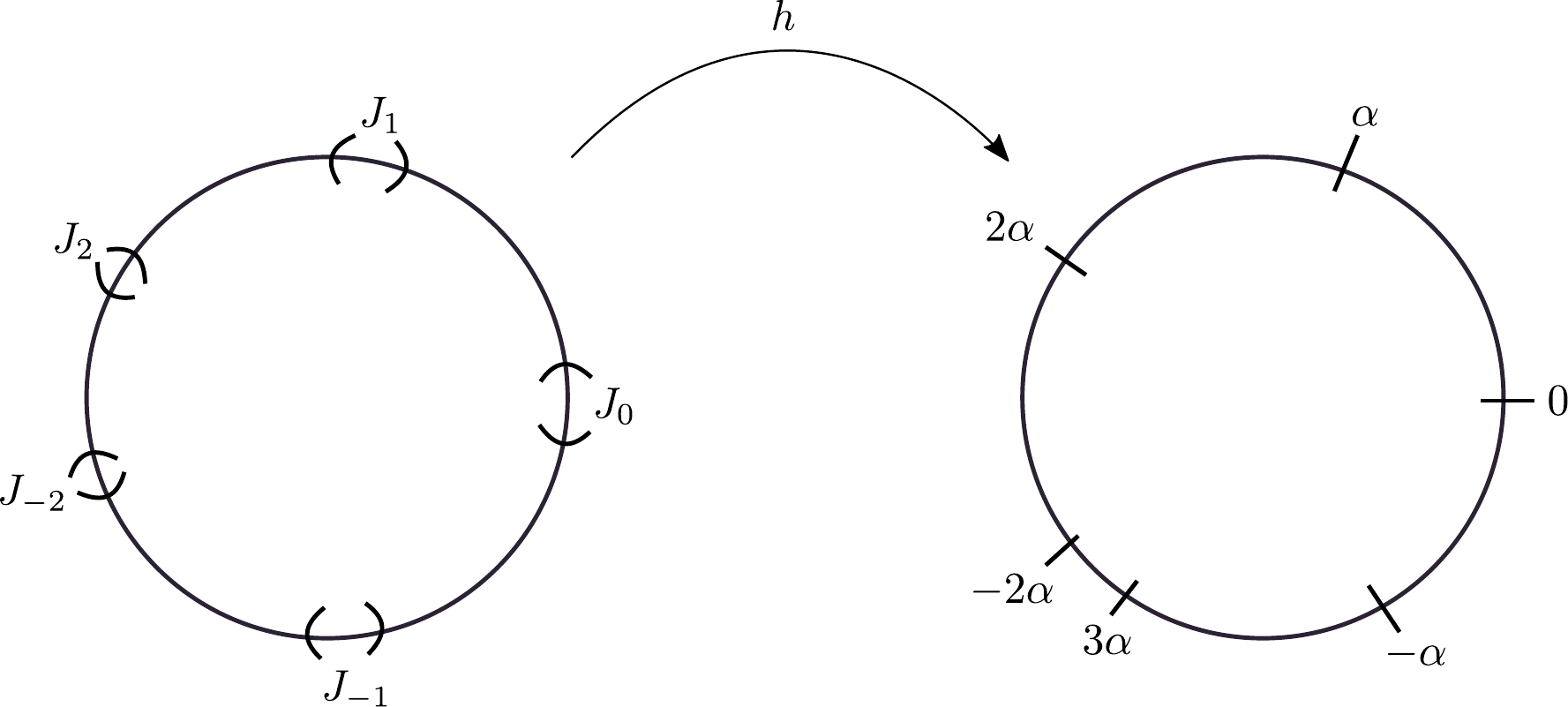}
\caption{Denjoy map construction}\label{fig:map}
\end{figure}

Then we map $J_n$ diffeomorphically onto $J_{n+1}$ with derivative equal to $1$ at the endpoints. This map has a continuous extension, which we denote by $f$, to the entire circle $S^1$, and, because of \eqref{eq:denj2}, this extension can be chosen to be of the class $\mathcal{C}^{1+\delta}$.

The minimal set (whose Hausdorff dimension and Hausdorff measure we want to estimate) of the map $f$ is then equal to
\begin{equation}
\Omega=\Omega_\alpha^\delta= S^1\setminus \bigcup_{n\in \Z} J_n.
\label{eq:defmin}
\end{equation}
Observe that $\Omega$ also depends on the choice of the sequence $(\ell_n)_{n\in \Z}$.
\begin{rem}
We want to note that in the construction of a Denjoy map $f$, the sum of lengths of $\ell_n$, $n\in \Z$, might actually be smaller than $1$, cf. \eqref{eq:denj1}. However, in such a situation the minimal set of $f$ would have Hausdorff dimension trivially equal to $1$ and positive Lebesgue measure. This is why we have removed such sequences from our considerations.
\end{rem}

\begin{rem} Observe that every Denjoy map can be obtained by such a construction. The non-existence of the full conjugation gives rise to the minimal set, whose \emph{omitted} intervals must follow the rotation by $\alpha$. Take the longest such interval, call it $J_0$, and then the semi-conjugation with the rotation about $\alpha$, defines the sequence $(\ell_n)_{n\in \Z}$. The property \eqref{eq:denj2} comes from the smoothness of the diffeomorphism. As we mentioned in the previous remark, the lengths $\ell_n$ need not sum up to $1$, but -- again -- we remove such case from our consideration.
\end{rem}

 We will denote by $h\colon\Omega\to S^1$ the semi-conjugacy given by Poincar\'e's Theorem. By its construction, this semi-conjugacy cannot be a (full) conjugacy. Nevertheless, $h^{-1}(z)$ is a singleton whenever $z$ does not belong to the orbit of $0$ under the rotation $R_\alpha$, i.e. whenever 
 $$
 z\notin \mathcal O_{R_\alpha}:=\{R^n(\alpha):n\in\Z\}.
 $$
Otherwise, $h^{-1}(z)$ is a doubleton. Furthermore, for any two points $x,y\in S^1\setminus O_{R_\alpha}$, we also have that
\begin{equation}
d(h^{-1}(x),h^{-1}(y))= \sum_{n\in\Z:R_\alpha^n(0)\in(x,y)}|J_n|,
\label{eq:distfor}
\end{equation}
where $|J_n|$ denotes the length of interval $J_n$ and $(x,y)$ is the shorter of the two arcs connecting $x$ and $y$; if the intervals have equal length, take the one in positive orientation. 

In general, i.e. for any points $x,y\in S^1$, the above formula actually still holds. More precisely, it holds for those respective preimages of $x$ and $y$ for which the arc connecting them is the shortest.

\subsection{Hausdorff Measure on Denjoy Minimal Sets} \ \\

As we are working with subtle measure estimates we deemed it prudent to put here the precise definitions in use in this paper.

First, recall our circle has length 1. We will use the standard version of the definition of the Hausdorff measure.
\begin{defin}The outer measure is the following
\[H_\beta(Y)= \lim_{r\to 0} \inf\left\{\sum_{k=1}^\infty (\diam U_k)^\beta : \forall_k \diam U_k <r \mbox{ and } Y\subset \bigcup_{k=1}^\infty U_k\right\},\]
where the infimum is take over all countable covers of $Y$ satisfying the conditions as stated. By Carath\'eodory's extension this gives the (typical) Hausdorff measure. 
\end{defin}
The next definition is also standard.
\begin{defin}
The Hausdorff dimension of the set $Y$ is given by the formula
\[\dim_H(Y)=\inf\{\beta \geq 0 : H_\beta(Y)=0\}.\]
\end{defin}

\subsection{Diophantine Properties} \ \\

We denote the standard continued fraction expansion of an irrational number $\alpha\in (0,1) $ as $[a_1, a_2, \ldots]$. The corresponding convergents are denoted as $\frac{p_n}{q_n}$. Recall that the convergents are the subsequent closest approximations of $\alpha$. Also, the following properties hold:
\begin{equation}\label{eq:confr1}
	q_{n}=a_n q_{n-1}+q_{n-2}, 
	\end{equation}
where $q_{-1}=0$ and $q_0=1$,	and 
\begin{equation}\label{eq:confr2}
	\frac{1}{(a_{n+1}+2)q_n}<\frac{1}{q_{n+1}+q_n}< \|q_n\alpha\|<\frac{1}{q_{n+1}}<\frac{1}{a_{n+1}q_n}.
\end{equation}
We now bring up the following classical definition.

	\begin{defin}
	An irrational number $\alpha\in(0,1)$ is said to be of \emph{Diophantine class} $\nu\in[1,+\infty]$, if the inequality
	\begin{equation}
	\|q\alpha\| < \frac{1}{q^\mu}.
	\label{eq:defdiophcl}
	\end{equation}
has infinitely many positive integral solutions $q$ for every $\mu<\nu$ and at only finitely many for every $\mu>\nu$. A number of Diophantine class $\nu=\infty$ is commonly called a Liouville number.
	\end{defin}
	
Recall that the set of irrational numbers of Diophantine class $\nu=1$ is of full Lebesgue measure. However, for any $\nu\in[1,+\infty]$ there exist $\alpha$ of Diophantine class $\nu$.

\begin{rem}
Observe that if $\nu>1$, then for infinitely many $n$ we have a very fast growth of~$q_n$'s. Combining \eqref{eq:confr2} and \eqref{eq:defdiophcl} gives $q_{n+1}\approx q_n^\nu$ for infinitely many $n$s. %Leading to a double exponential growth rate of $q_n\approx c^{\nu^n}$.
\end{rem}

\subsection{Recurrence Results} \ \\

To estimate the Hausdorff measure and Hausdorff dimension of the minimal set $\Omega$, we will apply the method described in a paper by the first named author, \cite{Pawpre}.

The result we need is this.

\begin{thm}
If $(X,d)$ is a metric space, $\mu$ is a Borel probability measure on $X$, and $T\colon X\to X$ is a Borel measurable map preserving measure $\mu$, then for every $\beta>0$ we have that
\begin{align}
&\liminf_{n\to \infty} nd^\beta (T^n(x),x))\leq g(x):=\limsup\limits_{r\to 0}\frac{H_\beta(B(x,r))}{\mu(B(x,r))}
\label{eq:thmrec}
\end{align}
\end{thm}
This result is applied in the following way. For $\mu$-almost, but see the comment in the next paragraph, every $x\in X$ find a lower bound of the recurrence speed, i.e. the left hand side of the inequality above. This gives a lower bound on the above function $g$, giving in turn a lower bound on the Hausdorff measure of $X$ by using the formula
\[H_\beta(X)=\int_X g(x) d\mu,\]
and the fact that $\mu(X)=1$.

Note that we do not have to really care about what the invariant measure actually is $X$ is compact and the map $T\colon X\to X$ is continuous. Such measure then exists because of Bogolubov-Krylov Theorem. Indeed, ff we do know what the invariant measure is, then we may do the calculations for almost every point with respect to this invariant. If not, then we need to do the calculations for all points; except countably many of them, if the measure is non-atomic.

\subsection{The Main Results} \ \\

As the general statement of our estimating method is rather complicated, we will state a few simpler versions. Also, our method can be applied under weaker assumptions, ex. for some sequences $(\ell_n)_{n\in\Z}$ not satisfying the requirements of Definition~\ref{def:denjoy}; see Example~ \ref{ex:expsubseq}.
\begin{thm}\label{520221110}
Fix $\delta\in(0,1)$ and irrational rotation number $\alpha$. 
Consider the corresponding Denjoy map $f\colon S^1\to S^1$ of the circle $S^1$ that is semi-conjugate to the rotation by $R_\alpha$ with gaps $J_n$, $n\in\Z$ whose sizes $|J_n|=\ell_n$ satisfy formulas \eqref{eq:denj1} and \eqref{eq:denj2} of Definition~\ref{def:denjoy}. For every $n\in\N$, set 
\[
N_n:=\left[\frac{q_n+q_{n+1}}{2}\right].
\]
Then for any $\beta>0$ we have that
\begin{equation}
H_\beta(\Omega_f)\geq \liminf_{n\to \infty} q_n\min{}^\beta\big\{|J_i|:-N_n\leq i \leq N_n\big\}.
\label{eq:Hausformula}
\end{equation}
In fact, the following more complicated but stronger, estimate holds.
\begin{equation}
H_\beta(\Omega_f)
\geq \liminf_{n\to \infty}q_n\left(\sum_{l=-\infty}^{+\infty}\min\big\{|J_i|:(2l-1)N_n\leq i \leq (2l+1)N_n\big\}\right)^\beta.
\label{eq:Hausformulb}
\end{equation}
Th strongest estimate we prove is the following. For every $n\in\N$, set $Q_n:=q_n+q_{n+1}$ and denote by $\mathcal{P}^n$ the set of all sequences 
$(P^{(n)}_l)_{l\in\Z}$ that satisfy the formula $P^{(n)}_l-P^{(n)}_{l-1}=Q_n$. Then
\begin{equation} 
H_\beta(\Omega_f)
\geq \liminf_{n\to \infty} q_n\sup{}^\beta\left\{\sum_{l=-\infty}^{+\infty} \min\big\{|J_i|:P^{(n)}_l\leq i < P^{(n)}_{l-1}\big\}:P^{(n)}\in \mathcal{P}^n\right\}.
\label{eq:Hausformulc}
\end{equation}

\end{thm}
A direct straightforward application of the first estimate in the above Theorem~\ref{520221110} gives the following.
\begin{cor}\label{cor:dimform}
With the notation and hypotheses of Theorem~\ref{520221110}, we have that
$$
\dim_H(\Omega_f)\geq \sup\left\{\beta\in\R^+: \liminf_{n\to \infty}q_n\bigg(\sum_{l=-\infty}^{+\infty}\min\big\{|J_i|:(2l-1)N_n\leq i \leq (2l+1)N_n\big\}\bigg)^\beta>0\right\}
$$
and
$$
\dim_H(\Omega_f)\geq \inf\left\{\beta\in\R^+:\liminf_{n\to \infty}q_n\bigg(\sum_{l=-\infty}^{+\infty}\min\big\{|J_i|:(2l-1)N_n\leq i \leq (2l+1)N_n\big\}\bigg)^\beta=0\right\}.
$$
\end{cor}
Lastly, with some additional work, as a consequence of the above corollary, we will prove (in Section \ref{sec:ex}) the lower bound on the Hausdorff dimension of $\Omega_f$ which was obtained first by Kra and Schmelling in \cite{KS} with an entirely different method. 
\begin{prop}\label{thm:dimbound}
Under assumptions as above 
\begin{equation}
\dim_H(\Omega_f)\geq \frac{\delta}{\nu}.
\label{eq:dimniceform}
\end{equation}
\end{prop}
Let us also observe that if the Diophantine class $\nu$ of the rotation number $\alpha$ is equal to $1$, then we may easily get estimates from above on the Haudorff dimension (and in some cases the Hausdorff measure).
\begin{thm}\label{thm:upbound}
For $\nu=1$ and the classic Denjoy sequence (see Example \ref{ex:classic}) the Hausdorff dimension equals $\delta$ and the relevant Hausdorff measure is positive and finite.
\end{thm}

\section{Proof of Theorem~\ref{520221110}}
Fix an irrational number $\alpha\in (0,1)$. Take a point $x_0$ in $\Omega_f$. For convenience assume that $x_0\notin \mathcal O_{R_\alpha}$. As we have already remarked, we may ignore all points in $\mathcal O_{R_\alpha}$ since this set is countable. In fact, our estimates also work for such points, but then one needs to be wary of the  somewhat annoying double preimages. 

We are to estimate from below the number $\displaystyle \liminf_{n\to \infty} n\left(d(T^n(x_0),x_0)\right)^\beta$. Using \eqref{eq:distfor}, we see that 
\begin{equation} \label{eq:metric}
\displaystyle d(f^n(x_0),x_0) = \sum_{n}|J_n|,
\end{equation} where the sum is taken over all $n$, for which $n\alpha\in(h(x_0),h(f^n(x_0)))$. 

To find the lower limit we are only interested in the sequence of subsequent closest returns. Since the map $h$ semi-conjugates the dynamics on $\Omega_f$ with the rotation by $\alpha$ on the circle, the closest returns on $\Omega_f$ are the closest returns for the rotation.
%\textbf{Tutaj uzywamy tego ze $l_n$ sa malejace oraz ze $l_n=l_{-n}$ albo potrzebujemy czegos podobnego, bo inaczej mogloby sie okazac, ze na okregu jestesmy ''daleko'', ale $J_n$ sa bardzo male, a gdy jestesmy blisko (ale 'po drugiej stronie $x_0$) to sa wieksze.}
They in turn are given by the reducts of the continued fraction expansion of $\alpha$. Thus, we only need to estimate the numbers
$$
d(f^{q_n}(x_0),x_0).
$$ 

Let us denote the $n$-th closest return of $x_0$ as $x_n=f^{q_n}(x_0)$. First, we estimate the sum of lengths \eqref{eq:metric} from below by taking only the biggest element.
\begin{align}
d(x_n,x_0)= \sum \left|J_{k}\right|\geq \left|J_{k_n}\right|,
\label{eq:defJkn}
\end{align}
where $k_n\in \Z$ is such that $k_n\alpha \in\big(h(x_0),h(x_n)\big)$ and the number $|J_{k_n}|$ is biggest possible (\emph{see Fig. \ref{fig:kn}}).
\begin{figure}[h]
	\centering
\includegraphics[scale=0.9]{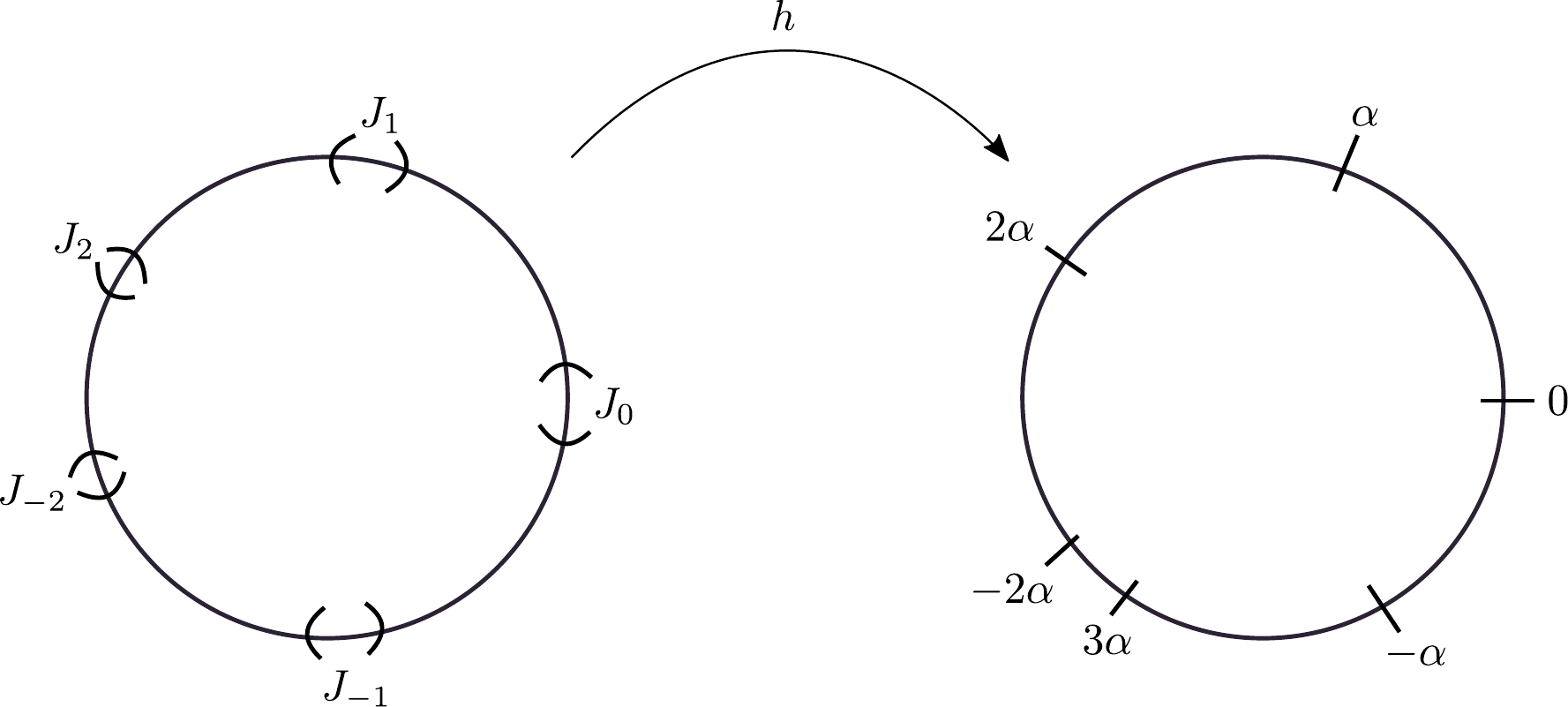}
\caption{Definition of the sequence $k_n$}\label{fig:kn}
\end{figure}

Now, we need to find bounds on $k_n$. By the semi-conjugacy $h(x_n)=h(x_0)+q_n\alpha$, and the search for $k_n$ may by restated as follows.

\emph{What is the smallest $k\in \N$, such that any interval on the circle of length $q_n\alpha$ contains at least one element of the set }
\[
\big\{-k\alpha, -(k-1)\alpha, \ldots, 0, \ldots, (k-1)\alpha,k\alpha\big\}?
\]
Since we are assuming that $h(x_0)$ is not in orbit of $0$ under $R_\alpha$, the gap between $h(x_n)$ and $h(x_0)$ cannot be equal to any of the above gaps. The answer to our above question is given in the following lemma. 

\begin{lem}\label{lem:frac}
The maximal gap between points in the set $\{t\alpha:t=-N\ldots 0\ldots N\}$ is at most of length $\|q_n\alpha\|$ if and only if 
$$
N\geq \left[\frac{q_n+q_{n+1}}{2}\right].
$$
\end{lem}
\begin{proof}
We use the three gaps theorem as stated in \cite{Rav}. Actually, we only need Lemma 2.1. from this paper, which (with the notation of this paper) states the following.
\begin{quote}
Take a forward orbit of a point under irrational rotation $R_\alpha$. For $K=q_n+q_{n+1}-1$ the set 
$$
\{\alpha k: k=0,\dots, K\}
$$ 
partitions the circle into $q_n$ gaps of length $\|q_{n+1}\alpha\|$ and $q_{n+1}$ gaps of length $\|q_n\alpha\|$. %, where $||x||$ is the distance between $x$ and the nearest integer.
\end{quote}
The figure below illustrates this phenomenon. We took $\alpha=\sqrt{3}-1$ and drew the first $25$ iterates. Note that all the gaps are only of two lengths.

\begin{figure}[h]
	\centering
\includegraphics[scale=0.4]{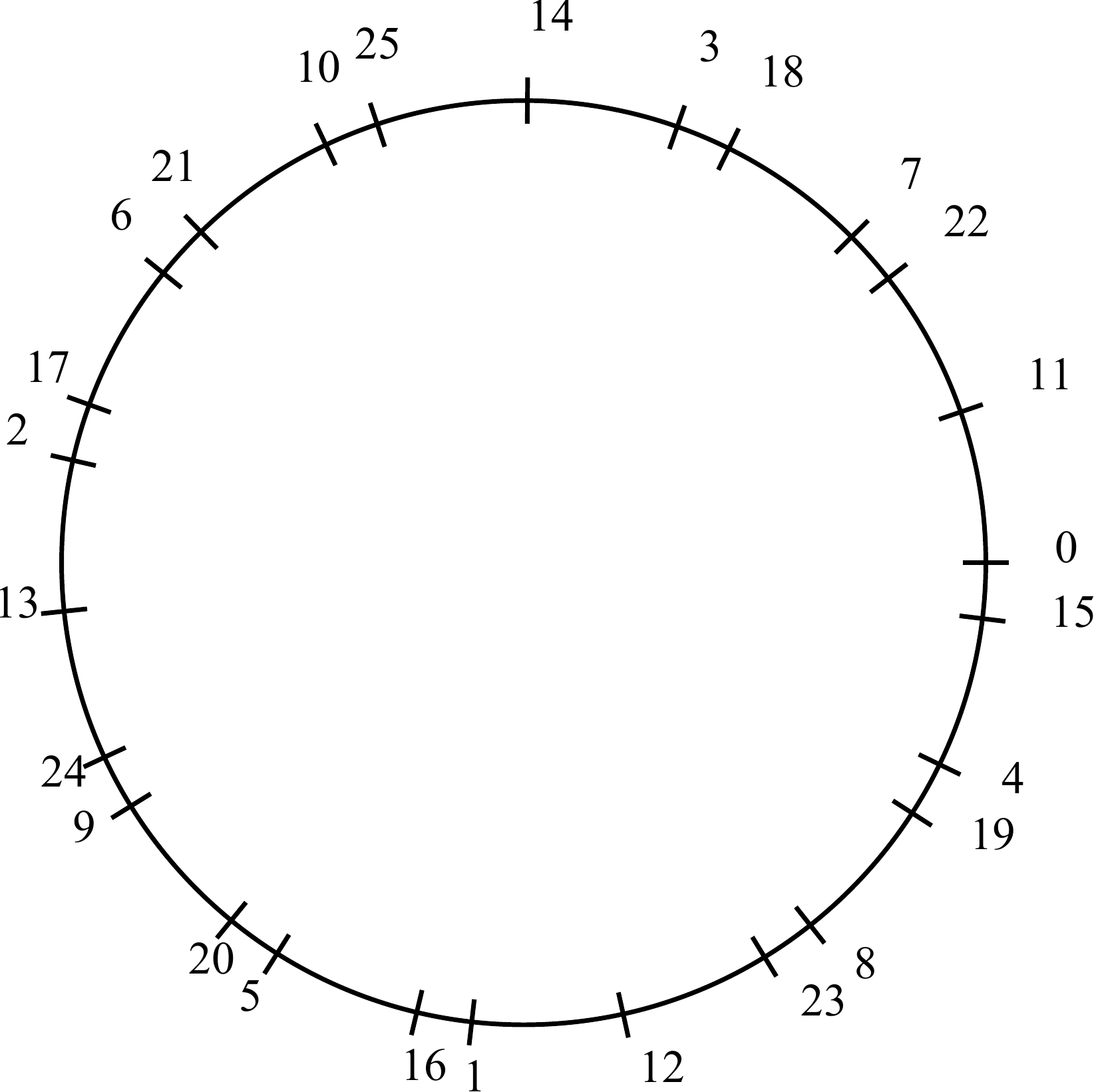}
\caption{Forward orbit of $0$ for $\alpha= \sqrt{3}-1$. }\label{fig:orblem1}
\end{figure}
In our setting we are working with $\displaystyle (t\alpha)_{t=-N}^N$. Rotating the circle (and all our points) by $N\alpha$, gives the set $\displaystyle (t\alpha)_{t=0}^{2N}$ and does not change the gap sizes, leading straight to the required result. 
The next figure shows the forward and backward orbit for the same $\alpha$. 

\begin{figure}[h!]
	\centering
\includegraphics[scale=0.4]{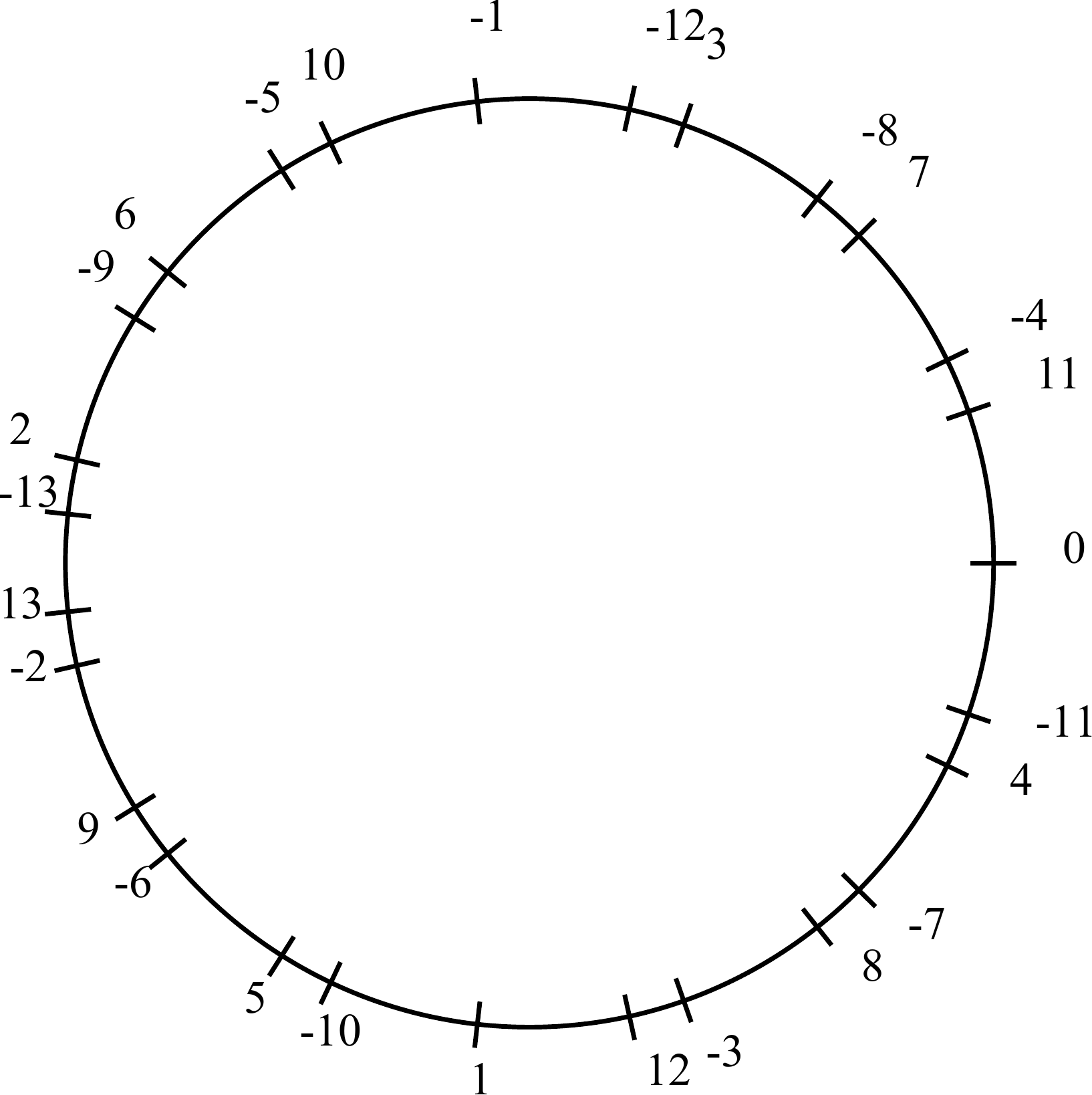}
\caption{Forward and backward orbit of $0$ for $\alpha= \sqrt{3}-1$. }\label{fig:orblem2}
\end{figure}
Note that there is one gap \emph{on the left of the circle} that is of a different size. This is the one resulting from the fact that such type of orbit (forward and backward) always has an odd number of points necessitating the rounding in our result. Removing $13$ gives the same picture as before but rotated in such a way that $-13$ becomes $0$. 
\end{proof}

Recall that we defined $N_n=\left[\frac{q_n+q_{n+1}}{2}\right]$. By Lemma~\ref{lem:frac}, we know that $|k_n|\leq N_n$, and we may estimate the size of $J_{k_n}$ from below as 
\begin{equation}
\left|J_{k_n}\right| \geq \min\big\{|J_i|:-N_n\leq i \leq N_n\big\}.
	\label{eq:sizeJkn}
\end{equation}
Inserting this into our estimate yields
\begin{equation}
H_\beta(\Omega)
\geq \liminf_{n\to \infty} nd^\beta(T^n(x_0),x_0)
\geq \liminf_{n\to \infty} q_n \min{}^\beta\big\{|J_i|:-N_n\leq i \leq N_n\Big\}.
\label{eq:Hausest}
\end{equation}
This ends the proof of \eqref{eq:Hausformula}.

To get better bounds we need to utilize all the elements of the sum in \eqref{eq:metric}.
Recall that the finite sequence $k\alpha$, $k\in [-N_n,\ldots,N_n]$, hits every gap of size $\|q_n\alpha\|$. Similarly, the sequence $k\alpha$, $k\in [N_n,\ldots,3N_n]$, hits every gap of the same size, and so does every sequence of type $k\alpha$, $[(2l-1)N_n,\ldots,(2l+1)N_n]$, for every integer $l$. Summing the lengths of the smallest intervals $J_i$ (on every sequence), leads straightforward to the estimate \eqref{eq:Hausformulb}.
\[H_\beta(\Omega)\geq \liminf_{n\to \infty} q_n\left(\sum_{l=-\infty}^{+\infty}\min\big\{|J_i|:(2l-1)N_n\leq i \leq (2l+1)N_n\big\}\right)^\beta.
\]
It is now easy to observe that our finite sequence may be \emph{moved around}, i.e. for any $P\in \Z$ the sequence $k\alpha$, for $k\in [P, P+Q_n)$ hits every gap of size $\|q_n\alpha\|$. We used $Q_n$ instead of $N_n=[Q_n/2]$ since there is no need for the finite sequences to be \emph{symmetric} w.r.t. zero. Also, removing the `integer part function' improves the estimate a little bit. 

It remains to divide for every $n\in\N$ the set of all integers into sets of length $Q_n$; in the statement of \eqref{eq:Hausformulc}, we called this `dividing' sequence $\left(P_l^{(n)}\right)_{l\in\Z}$. Then we find the minimal length of the intervals $|J_i|$ on every finite sequence, i.e. between $P_l^{(n)}$ and $P_{l-1}^{(n)}$, and we sum those values. The proves the general estimate \eqref{eq:Hausformulc}. 

\section{Examples} \label{sec:ex}
\begin{ex}\label{ex:clDen} Put $\alpha:=\frac{\sqrt{5}-1}{2}$ and
consider the classical Denjoy sequence 
\[
|J_n|=c_\delta (|n|+1)^{-1/\delta}, \ n\in\Z, 
\]
where, we recall, $c_\delta$ is the normalizing constant. Then the minimum in \eqref{eq:Hausformula} is attained at the largest possible $i$, i.e. at $i=N_n$, and, putting $\alpha_*:=\frac{1+\sqrt{5}}{2}$ our simpler estimate gives 
\begin{align*}
	H_\beta(\Omega)&\geq \liminf_{n\to \infty} c_\delta^\beta q_n \Big(\frac{1}{2}(q_n+q_{n+1}+1)+1\Big)^{-\beta/\delta} \geq \liminf_{n\to \infty} c_\delta^\beta q_n \Big(\frac{q_{n+2}}{2}+2\Big)^{-\beta/\delta}\\
	&\geq \liminf_{n\to \infty} \frac{c_\delta^\beta}{2^{-\beta/\delta}} \left(\frac{\alpha_*^{n}-1}{\sqrt{5}}\right) \left(\frac{\alpha_*^{n+2}+1}{\sqrt{5}}\right)^{-\beta/\delta}\\
	&\geq \frac{c_\delta^\beta\alpha_*^{-2\beta/\delta}}{2^{-\beta/\delta}5^{(\delta-\beta)/2\delta}}\lim_{n\to\infty} \alpha_*^{n(\delta-\beta)/\delta}.
\end{align*}
Therefore, $H_\beta(\Omega)>0$ for $\beta\in(0,\delta]$, thus proving that $\dim_H(\Omega)\geq \delta$, and we get the following estimate on the Hausdorff measure 
\[
H_\delta(\Omega) \geq 8c_\delta^\delta(1+\sqrt{5})^{-2}, 
\]
where $c_\delta=\sum_{n\in\Z} (|n|+1)^{-1/\delta}$.

We may easily improve the estimate by using the more complicated inequality \eqref{eq:Hausformulb}. The minima in the sum are attained at the endpoints of the intervals, i.e. at $(2k\pm1)N_n$ (where the sign depends on whether $k$ is positive or not). Let us first estimate the sum that occurs in the lower limit
%\begin{align*}
%H_\beta(\Omega)&\geq \liminf_{n\to \infty} q_n \left(\sum_{k=-\infty}^{+\infty} |J_{(2k\pm 1)N_n}|\right)^\beta\\&= \liminf_{n\to \infty} c_\delta^\beta q_n \left( (N_n+1)^{-1/\delta} + 2\sum_{k=1}^{+\infty} ((2k+1)N_n+1)^{-1/\delta}\right)^\beta\\
%&\geq \liminf_{n\to \infty} c_\delta^\beta q_n  (N_n+1)^{-\beta/\delta}\left(1 + 2\sum_{k=1}^{+\infty} (2k+1)^{-1/\delta}\right)^\beta
%\end{align*}
\begin{align*}
 \left(\sum_{k=-\infty}^{+\infty} |J_{(2k\pm 1)N_n}|\right)^\beta&=  c_\delta^\beta  \left( (N_n+1)^{-1/\delta} + 2\sum_{k=1}^{+\infty} \big((2k+1)N_n+1\big)^{-1/\delta}\right)^\beta\\
&\geq  c_\delta^\beta  (N_n+1)^{-\beta/\delta}\left(1 + 2\sum_{k=1}^{+\infty} (2k+1)^{-1/\delta}\right)^\beta \\&= c_\delta^\beta  (N_n+1)^{-\beta/\delta}\left(1+ 2\Big(\zeta\Big(\frac{1}{\delta}\Big)-1-2^{-1/\delta}\zeta\Big(\frac{1}{\delta}\Big)\Big)\right)^\beta,
\end{align*}
where we utilised the Riemann $\zeta$ function to simplify the harmonic sums. We used the following trivial equality:
\[\zeta(s) = \sum_{n=1}^\infty \frac{1}{n^s}=\sum_{k=1}^\infty \frac{1}{(2k)^s}+\sum_{k=0}^\infty \frac{1}{(2k+1)^s} = 2^{-s}\zeta(s)+\sum_{k=1}^\infty \frac{1}{(2k+1)^s}+1\]

Finally, we see that, the only difference in comparison to the previous estimate is the last (rather complicated) constant. Thus we may repeat the same calculation, getting at the end $\beta=\delta$. Let us also use the `closed form' of $c_\delta$ \eqref{eq:cdelta} arriving at
\[
H_\delta(\Omega) \geq 8\left(2\zeta\Big(\frac{1}{\delta}\Big)-1\right)^\delta(1+\sqrt{5})^{-2}\left(1+ 2\Big(\zeta\Big(\frac{1}{\delta}\Big)-1-2^{-1/\delta}\zeta\Big(\frac{1}{\delta}\Big)\Big)\right)^\delta. 
\]
\end{ex}

\begin{ex}\label{ex:expsubseq}
Keep $\alpha=\frac{\sqrt{5}-1}{2}$ but change the sequence $|J_n|$, $ n\in\Z$, slightly. Start with $\ell_n=(|n|+1)^{-1/\delta}$ but change the elements $\ell_{\pm4^k}:= 7^{-k}$, $k\in\N$. Normalize the sum by $c$ and put $|J_n|:=c \ell_n$.

\emph{Important note:} this does not satisfy \eqref{eq:denj2}! That is, our construction does not give a diffeomorphism on the entire circle -- only a homeomorphism. However, the method is still applicable.

Now, using our estimate in the straightforward manner by taking tiny gaps as the minimum `pollutes' the gap sizes, gives a weak result. 

However, recall that the sequences $k\alpha$, $k\in [-N_n,\ldots,N_n]$ and $k\alpha$, $k\in [N_n,\ldots,3N_n]$, hit every gap of size $\|q_n\alpha\|$. So, the sequence  $k\alpha$, $k\in [N_n,\ldots,5N_n]$, hits every gap at least twice. Now observe that the growth of the sequence $N_n\approx \alpha_*^n\leq 1.7^n$, $n\in\N$, is much slower than the growth of the distances between numbers of gaps with \emph{bad} sizes, i.e. $4^k$, $k\in\N$. Thus our set $[N_n,\ldots,5N_n]$ for $n$ large enough may hit at most one element of the sequence $4^k$, $k\in\N$. Since every gap is hit at least twice, we may ignore that one, and recover the estimate as before, only slightly worsened.
\begin{align*}
	H_\beta(\Omega)&\geq \liminf_{n\to \infty} c_\delta^\beta q_n \Big(\frac{5}{2}(q_n+q_{n+1}+1)+1\Big)^{-\beta/\delta} \geq \liminf_{n\to \infty} c_\delta^\beta q_n \Big(\frac{5q_{n+2}}{2}+2\Big)^{-\beta/\delta}\\
	&\geq 5^{-\beta/\delta}\frac{5^{-\beta/\delta}c_\delta^\beta\alpha_*^{-2\beta/\delta}}{2^{-\beta/\delta}5^{(\delta-\beta)/2\delta}}\lim_{n\to\infty} \alpha_*^{n(\delta-\beta)/\delta}.
\end{align*}
Therefore, $H_\beta(\Omega)>0$ for $\beta\in(0,\delta]$, thus proving that $\dim_H(\Omega)\geq \delta$, and we get the following estimate on the Hausdorff measure  
\[
H_\delta(\Omega) \geq \frac{8}{5}c_\delta^\delta(1+\sqrt{5})^{-2}.
\]
As above one could improve this estimate by using the whole sum, i.e. applying \eqref{eq:Hausformulc}.
\end{ex}

\begin{ex}
Take $\alpha\in(0,1)$ whose continued fraction expansion satisfies the relation 
\[
q_{n+1}=q_n^2
\] 
for every $n\in\N$. Then 
\[
a_n\approx c^{2^n}
\] 
for some integer $c\geq 2$. $\alpha$ is then of Diophantine class $\nu=2$. Set
\[
|J_n|:=cn^{-3}\ln|n|,
\]
where $c$ is the corresponding normalizing constant. 

We know that for such quickly growing numbers $q_n$ we have for any $\varepsilon>0$ and all integers $n\ge 1$ large enough that 
\begin{align}
	\frac{q_{n+1}+q_n}{2}<\frac{q_{n+1}}{2-\varepsilon}=\frac{q_n^2}{2-\varepsilon}.
\end{align}
Using this formula in our estimate, i.e. the computation as in the previous example, leads in the current case to the following.
$$
\begin{aligned}
H_\beta(\Omega)
&\geq \liminf_{n\to \infty} q_n\frac{c^\beta(\ln|q_n^2|-\ln|2-\varepsilon|)}{q_n^{6\beta}(2-\varepsilon)^{-3}} \\
&=c^\beta(2-\varepsilon)^3\liminf_{n\to \infty}q_n^{1-6\beta} (\ln|q_n^2|-\ln|2-\varepsilon|).
\end{aligned}
$$
We see from this inequality that if $\beta\leq \frac{1}{6}$, then $H_\beta(\Omega)=+\infty$, giving the estimate \[\dim_H(\Omega)\geq 1/6.\]
\end{ex}
\begin{proof}[Proof of Proposition \ref{thm:dimbound}]
To get the desired lower bound on the Hausdorff dimension of $\Omega_f$, we need two ingredients: some kind of a universal bound on the lengths $|J_n|$, $n\in\Z$, and some bound on the growth of the denominators of the continued fraction expansion.

The assumption  \eqref{eq:denj2} imposed on the sequence $\ell_n=|J_n|$, $n\in\Z$, easily leads to the following result which was stated and proved as Lemma 2.3 in \cite{KS}.
\begin{quote}
For all $0<\theta<\delta<1$ there exists $s_\theta\ge 1$ such that
\begin{equation}
|J_n|>n^{-1/\theta}
\label{eq:KSLemma}
\end{equation}
for all integers $n>s_\theta$.
\end{quote}

Now we shall deal with the growth of the $q_n$s, $n\in\N$. Combining the continued fraction property \eqref{eq:confr2} and the Diophantine class definition \eqref{eq:defdiophcl} we get for every $\varepsilon>0$ and all $n\geq s_\varepsilon$  that
$$
\frac{1}{q_{n+1}}>\|q_n\alpha\|>\frac{1}{q_n^{\nu+\varepsilon}}.
$$
This gives the desired estimate 
\begin{equation}\label{eq:diophcomb}
q_{n+1}<q_n^{\nu+\varepsilon}.
\end{equation}

It remains to estimate below the lower limit
\begin{equation}
 \liminf_{n\to \infty}q_n\min{}^\beta\big\{|J_i|:-N_n\leq i \leq N_n\big\}.
\label{eq:proppf1}
\end{equation}

To do this, first fix $\theta\in(0,\delta)$ and also any $\varepsilon>0$. Then fix any integer $n>\max(n_\varepsilon, n_\theta)$. We will obtain the desired lower estimate by using first \eqref{eq:KSLemma}, then the definition of $N_n$ together with the fact $N_n<q_{n+1}$, and finally \eqref{eq:diophcomb}. With any $\beta>0$, all of these lead to
\begin{align*}
	\liminf_{n\to \infty} q_n \min{}^\beta\big\{|J_i|:-N_n\leq i \leq N_n\big\} 
	&\geq \liminf_{n\to \infty} q_n \Big( N_n^{-1/\theta}\Big)^\beta \geq \liminf_{n\to \infty} q_n \cdot q_{n+1}^{-\frac{\beta}{\theta}}\\ &\geq \liminf_{n\to \infty} q_n \cdot q_n^{-(\nu+\varepsilon)\frac{\beta}{\theta}}= \liminf_{n\to \infty} q_n^{1-\frac{\nu+\varepsilon}{\theta}\beta}.
\end{align*}
So, this lower limit is positive for $\beta=\frac{\theta}{\nu+\varepsilon}$. Therefore, by applying Corollary~\ref{cor:dimform}, we get that $\dim_H(\Omega)\geq \frac{\theta}{\nu+\varepsilon}$. Hence, by letting first $\varepsilon$ go to $0$ and then $\theta$ to $\delta$, we obtain that 
$$
\dim_H(\Omega)\geq \frac{\delta}{\nu}.
$$
\end{proof}
%Let us also observe that if the Diophantine class $\nu$ of the rotation number $\alpha$ is equal to $1$, then we may easily get estimates from above on the Haudorff dimension (and in some cases the Hausdorff measure).
%\begin{prop}
%For $\nu=1$ and the classic Denjoy sequence (see Example \ref{ex:clDen}) the Hausdorff dimension equals $\delta$ and the relevant Hausdorff measure is positive and finite.
%\end{prop}
\begin{proof}[Proof of Theorem \ref{thm:upbound}]
By the previous results, we only need to prove that $H_\delta(\Omega)<+\infty$. 
Observe the set \[S^1\setminus \bigcup_{k=-n}^n J_k\supset \Omega.\] 
It consists of $2n+1$ disjoint intervals, let us call them $I_p$, whose total length equals 
\[1-\sum_{k=-n}^n |J_k|=\sum_{k<-n}|J_k|+\sum_{k>n}|J_k|.\]
For the classic sequence both of the sums on the RHS are equal and may be trivially bounded by and integral giving
\[\sum_{|k|>n}|J_k|\leq 2\frac{\delta}{\delta-1}c_\delta (n-1)^{1-1/\delta}.\]
As $\delta<1$, function $x^\delta$ is concave and application of Jensen's inequality leads to the following bound on the covering sum (for clarity, put $N=2n+1$ until the last equality).
\[\sum_{p=1}^{N}|I_p|^\delta \leq N\left(\sum_{p=1}^{N}\frac{1}{N}|I_p|\right)^\delta\leq N^{1-\delta}\left(2\frac{\delta}{\delta-1}c_\delta (n-1)^{1-1/\delta}\right)^\delta=\left(2\frac{\delta}{\delta-1}c_\delta\right)^\delta\left(\frac{2n+1}{n-1}\right)^{1-\delta}.\]
Clearly, the lengths (diameters) of the intervals $|I_p|$ shrink to 0 as $n\to +\infty$. Thus, we arrive at
\[H_\delta (\Omega)\leq 2\left(\frac{\delta}{\delta-1}c_\delta\right)^\delta.\qedhere\]  
\end{proof}
\begin{rem}Note that the calculation above works for all irrational $\alpha$, but whenever $\nu>1$ it does not give the proper upper estimate on the dimension. It should not come as a surprise as in the proof above we essentially treat the cover as its elements were of comparable sizes. 

On the other hand, if $\nu>1$, then looking at the rotation by $\alpha$ sometimes the orbit returns very close to the starting point. So in the proof above some of the sets $I_p$ will have very small length in comparison with others. Thus, a cover were the sizes are similar may not be optimal.
\end{rem}
Following the proof of the previous result we get the following generalisation.
\begin{cor}
Assume $\nu=1$. If for every $\varepsilon>0$, there exists a subsequence $(n_s)$ and constant $C<\infty$ such that
\[ \sum_{|k|>n_s}|J_k|\leq C n_s^{1-\frac{1}{\delta+\varepsilon}},\]
then $\dim_H(\Omega)=\delta$.
\end{cor}
\section*{Acknowledgements}
The research of \L{}ukasz Pawelec was supported in part by National Science Centre, Poland, Grant OPUS21 ``Holomorphic dynamics, fractals, thermodynamic formalism'', 2021/41/B/ST1/00461.

The research of Mariusz Urbanski was supported in part by the Simons Grant: 581668.

%\bibliography{bib}

\end{document}